\documentclass{amsart}

\usepackage{amssymb}
\usepackage{amsmath}

\input{macroced}
\usepackage{todonotes}
\usepackage{mathrsfs}
\usepackage{color,cite,graphicx}
\usepackage{tikz}
\definecolor{gray}{gray}{0.4}
\usepackage[all]{xy}
\usepackage{enumerate}
\usepackage{array, boldline, makecell, booktabs}
\usepackage{multirow,multicol}

\title[K3 surfaces with maximal finite automorphism groups]{K3 surfaces with maximal finite automorphism groups 
containing ${\boldsymbol{M_{20}}}$}

\author{C\'edric Bonnaf\'e}
\address{IMAG, Universit\'e de Montpellier, CNRS, Montpellier, France} 

\email{cedric.bonnafe@umontpellier.fr}
\urladdr{http://imag.umontpellier.fr/~bonnafe/}
\makeatother

\author{Alessandra Sarti}
\address{Laboratoire de Math\'ematiques et Applications, UMR CNRS 7348,
			Universit\'e de Poitiers, T\'el\'eport 2, Boulevard Marie et Pierre Curie,
			86962 FUTUROSCOPE CHASSENEUIL, France}
\email{sarti@math.univ-poitiers.fr}
\urladdr{http://www-math.sp2mi.univ-poitiers.fr/~sarti/}

\date{\today}

\thanks{The first author is partly supported by the ANR: 
Projects No ANR-16-CE40-0010-01 (GeRepMod) and ANR-18-CE40-0024-02 (CATORE)}

\newtheorem{theorem}{Theorem}[section]
\newtheorem{pro}[theorem]{Proposition}
\newtheorem{lemma}[theorem]{Lemma}
\newtheorem{cor}[theorem]{Corollary}

\theoremstyle{remark}

\newtheorem{remark}[theorem]{Remark}

\DeclareMathOperator{\diag}{diag}
\DeclareMathOperator{\Km}{Km}

\renewcommand{\IC}{\mathbb{C}}

\newcommand{\IZ}{\mathbb{Z}}

\newcommand{\IP}{\mathbb{P}}

\newcommand{\lra}{\longrightarrow}

\def\kondo{{\mathrm{Ko}}}
\def\mukai{{\mathrm{Mu}}}
\def\bra{{\mathrm{BH}}}
\def\kummer{{\mathrm{Km}}}

\begin{document}

\maketitle

%\todo{Est-ce qu'on change le titre pour mettre $M_{20}$? 
%Par exemple : {\it K3 surfaces with maximal finite automorphism groups 
%containing $M_{20}$}}

\begin{center}
{\it In memory of Laurent Gruson}
\end{center}

\begin{abstract}
It was shown by Mukai that the maximum order of a finite group
acting faithfully and symplectically on a K3 surface is $960$ and that if such 
a group has order $960$, then it is isomorphic to the Mathieu group $M_{20}$. 
Then Kondo showed that the maximum order of a finite group acting faithfully 
on a K3 surface is $3\,840$ and this group contains 
$M_{20}$ with index four. Kondo also showed that there is a unique K3 surface
on which this group acts faithfully, which is the Kummer surface $\Km(E_i\times E_i)$.
In this paper we describe two more K3 surfaces admitting a big finite automorphism group
of order $1\,920$, both groups contains $M_{20}$ as a subgroup of index 2. We show moreover that
these two groups and the two K3 surfaces are unique. 
This result was shown independently 
by S. Brandhorst and K. Hashimoto in a forthcoming
paper, with the aim of classifying all the finite groups 
acting faithfully on K3 surfaces with maximal symplectic part. 
\end{abstract}

%\maketitle

\bigskip

\section{Introduction}
A K3 surface is a compact complex surface
which is simply connected and has trivial canonical bundle. 
Given a finite group $\G$ acting on a K3 surface $X$ 
we have an exact sequence
$$
1\lra \G_0\lra \G\lra \mathbb{Z}/m\mathbb{Z}\lra 1
$$
where the last map is induced by the action on the nowhere vanishing holomorphic 2-form $\omega_X$. 
The group $\G_0$ is  the normal subgroup of maximal order contained in $\G$ whose 
automorphisms act trivially
on $\omega_X$. The automorphisms of $\G_0$ are called {\it symplectic}. 
%The quotient $\G/\G_0\cong \mathbb{Z}/m\mathbb{Z}$
%is hence a cyclic group, and we say that its generator, acts {\it purely non--symplectically} on $X$.
It was shown by Mukai~\cite[Theorem 0.3]{mukai} that, if $G$ is a finite group 
acting faithfully and symplectically on a K3 surface, then $|G| \le  960$ and, 
if $|G|=960$, then $G$ is isomorphic to the Mathieu group 
$M_{20}$. In his paper Mukai gives the example
of a K3 surface with such an action, we recall this example 
in section~\ref{sec:g29}.
More generally, it is an interesting question to classify maximal finite groups 
$\G$ acting faithfully on a K3 surface. More precisely we say that $\G$ is a 
{\it maximal finite group} acting faithfully 
on a K3 surface if the following holds: assume $\G'$ is another finite group acting 
faithfully on a K3 surface
then $\G$ is not (isomorphic to) a proper subgroup of $\G'$. 

In Theorem~\ref{max} we show that there 
are only three finite groups $\G$ containing 
strictly $\Gamma_0=M_{20}$ as the normal subgroup of $\G$ acting faithfully and 
symplectically and only three K3 surfaces acted on by such a $\G$, 
the main ingredient of the proof is Theorem~\ref{involutions}. 
This result is shown also independently in a forthcoming paper of S. Brandhorst 
and K. Hashimoto~\cite{BH},
where they compute all the finite groups 
acting faithfully on K3 surfaces with maximal symplectic part. 
In our situation one of the three K3 surfaces mentioned above was 
constructed by Kondo~\cite{kondo} (this is the only K3 surface 
acted on faithfully by a finite group of order $3~\!840=4 \cdot |M_{20}|$), 
another one was constructed by Mukai~\cite{mukai}, 
 and the existence of the last one was showed by Brandhorst-Hashimoto 
 in {\it loc. cit.}, we give here explicit equations. 
In the second and in the third case the order of $\G$ is equal to  $2 \cdot |M_{20}|$. 
We denote these three surfaces respectively by 
$X_\kondo$, $X_\mukai$ and $X_\bra$. In this note, we  compute the transcendental lattice of these three 
K3 surfaces. This was done by Kondo for the surface $X_\kondo$, we recall it here 
to have a complete picture, and we compute it for $X_\mukai$ and $X_\bra$. Accordingly to \cite[Section 3]{dolga_quartic}
the transcendental lattice
of $X_\mukai$ was already known by Mukai, but we could not find explicit computations, so we give it here. 
We give also equations for the three
surfaces. Mukai already provided equations for $X_\mukai$ as a 
smooth quartic surface in $\IP^3(\IC)$ (which is the {\it  Maschke surface}, see \cite[Section 3]{dolga_quartic})  
we compute it here in a different way, but
we show that up to a projective transformation of $\IP^3(\IC)$, these are equivalent. 

The equations for 
$X_{\kondo}$ and $X_\bra$ are new. In particular one gets easily a (singular) equation for the first one as
a complete intersection of two quartics in weighted projective space $\IP(1,1,2,2,2)$ by using
a result of Inose, \cite{inose}. To get the equations for $X_\bra$ one needs a more careful 
study of 
the action of $M_{20}$ on the projective space $\IP^5(\IC)$. It turns out that $X_\bra$ is a smooth
complete intersection of three quadrics and we give here the equations (this answers a question of S. Brandhorst 
to the authors). All these three K3 surfaces turn out to be
Kummer surfaces of abelian surfaces that are the product of two elliptic
curves, see Corollary \ref{cor:parity}. By using results of Shioda and Mitani \cite{shiodamitani} we compute explicitly 
the two elliptic curves. We have that
\begin{eqnarray*}
\begin{array}{c}
X_\kondo\cong \kummer(E_i\times E_i),\qquad 
X_\mukai\cong \kummer(E_{i\sqrt{10}}\times E_{i\sqrt{10}}),\\ 
\\
X_\bra\cong \kummer(E_{\tau}\times E_{2\tau}),\quad\mbox{with}\, 
\tau=\frac{-1+i\sqrt{5}}{2}.\\
\end{array}
\end{eqnarray*}
Here, $E_z$ denotes the elliptic curve with complex multiplication 
given by $z$. 
For the example of $X_\bra$, we also obtain in Remark~\ref{curves bra} 
an explicit Nikulin configuration of $16$ disjoint smooth rational curves 
(we are not able 
to obtain such an explicit configuration for $X_\mukai$: see Remark~\ref{rem:coniques dans q4}).
% In Remarks~\ref{rem:coniques dans q4} and~\ref{curves bra} we discuss sets of disjoint rational curves on $X_\mukai$ and $X_\bra$.

\bigskip

\noindent{\bf Acknowledgements:} We warmly thank Simon Brandhorst for 
letting us know about his forthcoming paper with Kenji Hashimoto~\cite{BH}, for very 
useful discussions, and for a careful reading of a preliminary version of this 
work, J\'er\^ome Germoni who helped us finding the 16 disjoint 
smooth rational curves on $X_\bra$ thanks to his explanations on cliques, 
Xavier Roulleau for stimulating comments on an early version of the paper 
and the MSRI for letting the first author use his computing 
facilities: many hidden computations were done using {\sc Magma}~\cite{magma}. 
We finally thank the anonymous referee for a careful reading of the paper and for useful comments.

\vspace{1.0cm}

\noindent{\sc Notation - } 
If $G$ is a group, we denote by $G'$ its commutator subgroup (also sometimes called 
derived subgroup) 
and by $\Zrm(G)$ its center. If $V$ is a vector space, we denote by 
$\CM[V]$ the algebra of polynomial functions on $V$ and, if $k \ge 0$, 
we denote by $\CM[V]_k$ its homogeneous component of degree $k$. 
If $f_1$,\dots, $f_r \in \CM[V]$ are homogeneous, we denote by $\ZC(f_1,\dots,f_r)$ 
the associated scheme of $\PM(V)$, defined by $f_1=\cdots=f_r=0$. 
If $G$ is a subgroup of $\Gb\Lb_\CM(V)$, we denote by $PG$ its image 
in $\Pb\Gb\Lb_\CM(V)$. If $V=\CM^n$, we identify naturally $\Gb\Lb_\CM(V)$ 
and $\Gb\Lb_n(\CM)$. 
We denote by $M_{20}$ the Mathieu group of order $960$.

If $\t \in \CM$ has a positive imaginary part, we denote by $E_\t$ 
the elliptic curve $\CM/(\ZM \oplus \ZM\t)$. If $A$ is an abelian surface, 
we denote by $\kummer(A)$ its associated Kummer surface. 
We denote by $\Lb$ the {\it K3 lattice} 
$E_8(-1) \oplus E_8(-1) \oplus U \oplus U \oplus U$, 
where $U$ is the hyperbolic plane and $E_8(-1)$ is the lattice 
$E_8$ endowed with the opposite quadratic form. 
If $X$ is a K3 surface, we denote by $\Lb_X$ the lattice 
$\Hrm^2(X,\ZM)$ (it turns out that $\Lb_X \simeq \Lb$) 
and by $\Tb_X$ its transcendental lattice (i.e. the orthogonal, 
in $\Lb_X$, of its N\'eron-Severi group). Finally, 
we denote by $\Lb_{20}$ the lattice 
$$\Lb_{20}=
\begin{pmatrix}
4&0&-2\\
0&4&-2\\
-2&-2&12\\
\end{pmatrix}.$$
See the Proposition~\ref{prop:l20} below 
for the reason for this notation. 

\medskip

\section{K3 surfaces with a faithful action of $M_{20}$}

\medskip

We gather in this section some properties of the K3 surfaces 
admitting a faithful action of the finite group $M_{20}$ (since $M_{20}$ is equal 
to its commutator subgroup, this is necessarily a symplectic action), and we prove 
the main result of this paper, namely a classification of K3 surfaces admitting a 
faithful action of a finite group containing strictly $M_{20}$. 

If we consider all the K3 surfaces $X$ that admit a faithful symplectic 
action of $M_{20}$, Xiao~\cite[Nr. 81, Table 2]{xiao} proved that 
the minimal resolution of the quotient of 
$X$ by $M_{20}$ is a K3 surface with Picard number 20. By a result of 
Inose~\cite[Corollary 1.2]{inose}, 
this means also that $X$ has Picard number 20. This shows 
the following, with the same notation as before:

\medskip

\begin{pro}
There are at most countably many K3 surfaces with a 
faithful symplectic action by $M_{20}$.
\end{pro}

\medskip

\begin{proof}
Since the Picard number is 20, then the moduli space of K3 surfaces with a 
faithful symplectic $M_{20}$-action is 0-dimensional. 
\end{proof}

\medskip

\begin{remark}
Observe that the automorphism group of a K3 surface with  
Picard number $20$ is infinite~\cite[Theorem~5]{shiodainose}. Shioda and Inose 
show it by exhibiting an elliptic fibration with an infinite order section, 
this gives an automorphism acting symplectically on the K3 surface with 
infinite order.\finl
\end{remark}

\medskip

% We denote by $\Lb$ the {\it K3 lattice} 
% $E_8(-1) \oplus E_8(-1) \oplus U \oplus U \oplus U$, 
% where $U$ is the hyperbolic plane and $E_8(-1)$ is the lattice 
% $E_8$ endowed with the opposite quadratic form. 
% If $X$ is a K3 surface, we denote by $\Lb_X$ the lattice 
% $\Hrm^2(X,\ZM)$ (it turns out that $\Lb_X$ is isometric to $\Lb$) 
% and by $\Tb_X$ its transcendental lattice. Finally, 
% we denote by $\Lb_{20}$ the lattice 
% $$\Lb_{20}=
% \begin{pmatrix}
% 4&0&-2\\
% 0&4&-2\\
% -2&-2&12\\
% \end{pmatrix}.$$
Recall the following result~\cite[proof of Proposition~2.1]{kondo}:

\medskip

\begin{pro}\label{prop:l20}
Let $X$ be a K3 surface with a faithful symplectic action by $M_{20}$. Then 
the invariant lattice $\Lb_X^{M_{20}}$ is isometric to $\Lb_{20}$.
\end{pro}

\medskip

\begin{remark}\label{rem:l20}
Note that $\Lb_{20}$ has signature $(3,0)$, so its isometry group 
is finite. Let us recall its description. Let 
$$\r_1=\begin{pmatrix} 0 & 1 & 0 \\ 1 & 0 & 0 \\ 0 & 0 & 1 \end{pmatrix}
\qquad\text{and}\qquad 
\r_2=\begin{pmatrix} 1 & 0 & -1 \\ 0 & -1 & 0 \\ 0 & 0 & -1 \end{pmatrix}.$$
Then $\r_1$ and $\r_2$ belong to the group of isometries of $\Lb_{20}$ and 
it is easily checked that the group of isometries of $\Lb_{20}$ 
is generated by $\r_1$, $\r_2$ and $-\Id_{\Lb_{20}}$ (by using 
for instance the upcoming Lemma~\ref{lem:plongements}) 
and has order $16$ (see also~\cite[Proposition~2.1]{kondo}).\finl
\end{remark}

\medskip

\begin{cor}\label{cor:parity}
If a K3 surface $X$ admits a faithful action by the group 
$M_{20}$ then $X=\kummer(A)$ for a unique abelian surface $A$,
which is the product of two elliptic curves.
\end{cor}

\medskip

\begin{proof}
Let $(u,v)$ be a $\ZM$-basis of $\Tb_X \subset \Lb_X^{M_{20}}$. 
By Proposition~\ref{prop:l20}, we have $u^2$, $v^2 \in 4\ZM$ 
and $u \cdot v \in 2\ZM$. So
$$
\Tb_X \simeq \left(\begin{array}{cc}
4a&2b\\
2b&4c\\
\end{array}
\right).
$$
Following~\cite[Section~3]{shiodamitani}, we set $A\cong E_{\t_1} \times E_{\t_2}$ where 
$$
\tau_1=\frac{-b+\sqrt{\Delta}}{2a},\qquad \tau_2=\frac{b+\sqrt{\Delta}}{2}
$$
and $\Delta=b^2-4ac$, so that 
$$
\Tb_A:=\left(\begin{array}{cc}
2a&b\\
b&2c\\
\end{array}
\right).
$$
Hence $\Tb_X=\Tb_A(2)=\Tb_{\kummer(A)}$.

The uniqueness follows from~\cite[Theorem~5.1]{shiodamitani}.
\end{proof}

\bigskip

\begin{remark}\label{rem:indec}
Let us prove here that $\Lb_{20}$ is indecomposable. %\todo{C'est peut-etre bien connu, il faudrait essayer de chercher une reference} 
Assume that it is not indecomposable. Then 
$\Lb_{20}=L_1 \DS{\mathop{\oplus}^\perp} L_2$, 
where $L_1$ has rank $1$ and $L_2$ has rank $2$. By the proof of the 
Corollary~\ref{cor:parity}, we have $L_1 = \langle 4n \rangle$ for some 
$n \ge 0$ and 
$$L_2=\begin{pmatrix} 4a & 2b \\ 2b & 4c \end{pmatrix}$$
for some $a$, $b$, $c$ $\in\mathbb{Z}$. Then 
$160={\mathrm{disc}}(\Lb_{20})={\mathrm{disc}}(L_1){\mathrm{disc}}(L_2)
= 16n(4ac-b^2)$. 
In other words, $10=n(4ac-b^2)$, which means that $4ac-b^2 \in \{1,2,5,10\}$.
But $b^2 \equiv 0$ or $1 \mod 4$, so $4ac-b^2 \equiv 3$ or $4 \mod 4$. 
This leads to a contradiction.\finl
\end{remark}

%\todo{maybe cite here a thm of S. Brandhorst if there is (Cedric's comment: ok)}

% Observe that the K3 surfaces $X_\mukai$ and $X_\bra$ are the unique
% K3 surfaces with an action of the two index two extensions  $G_\mukai$ respectively
% $G_\bra$ of $M_{20}$. We show the following

Our main result in this paper is the following:

\medskip

\begin{theorem}\label{involutions}
Assume that $M_{20}$ acts faithfully on a K3 surface $X$, and assume moreover that $X$
admits a non--symplectic automorphism $\iota$
acting on it, normalizing $M_{20}$ and such that $\iota^2 \in M_{20}$.
We set $G=\langle \iota \rangle M_{20}$.
Then we have the following three possibilities for
the $G$-invariant N\'eron-Severi group of $X$ and its transcendental lattice:
\begin{enumerate}
\itemth{1} $\langle 40 \rangle, \qquad \left(\begin{array}{cc}
4&0\\
0&4\\
\end{array}
\right)$

\itemth{2} $\langle 4 \rangle, \qquad \left(\begin{array}{cc}
4&0\\
0&40\\
\end{array}
\right)$

\itemth{3} $\langle 8 \rangle, \qquad \left(\begin{array}{cc}
8&4\\
4&12\\
\end{array}
\right)$
\end{enumerate}
All the three cases are possible and are described in the sections \ref{sec:kondo},
\ref{sec:g29}, \ref{sec:q8}.
\end{theorem}

\medskip

\begin{proof}
We only prove here the fact that the N\'eron-Severi group of $X$ 
and its transcendental lattice is necessarily one of the given 
three forms: the existence of the three examples 
will be shown in the upcoming sections (and we will add some geometric 
features of those examples). We first need two technical lemmas:

\begin{quotation}
\begin{lemma}\label{lem:plongements}
Up to isometry, there is a unique embedding of the lattice $\langle 4 \rangle$ 
(resp. $\langle 8 \rangle$, resp. $\langle 40 \rangle$) as a primitive 
sublattice of $\Lb_{20}$. 
\end{lemma}

\begin{proof}[Proof of Lemma~\ref{lem:plongements}]
The uniqueness of the embedding of $\langle 40 \rangle$ is shown 
in~ \cite[Lemma 3.1]{kondo}. For the two other cases, let $(e,f,h)$ 
denote the canonical basis of the lattice $\Lb_{20}$ and let $L$ 
be a primitive element of $\Lb_{20}$ such that $L^2=4$ (resp. $8$). 
Write $L=\l e + \mu f + \d h$ with $\l$, $\mu$, $\d \in \ZM$. Then 
$$L^2=(2\l-\d)^2+(2\mu-\d)^2 + 10 \d^2,$$
so $\d=0$ and $\l^2+\mu^2 =1$ (resp. $\l^2+\mu^2=2$). 
This gives $(\l,\m)=(\pm 1,0)$ or $(0,\pm 1)$ (resp. $(\pm 1,\pm 1)$). 
So $L=\pm e$ or $\pm f$ (resp. $L=\pm e \pm f$), and the four solutions  
are in the orbit of the group $\langle -\Id_{\Lb_{20}},\r_1 \rangle$ 
(resp. $\langle -\Id_{\Lb_{20}},\r_2 \rangle$).
\end{proof}
\end{quotation}

We choose an isomorphism between $\Lb_{20}$ and $\Lb_X^{M_{20}}$. 
Then the group $G/M_{20}=\langle \iota \rangle$ acts on $\Lb_{20}$ 
and $\iota$ acts by $-\Id$ on $\Tb_X$. Also, 
the lattice $\Lb_X^G$ has rank $1$ because $\Tb_X$ has rank $2$.

\begin{quotation}
\begin{lemma}\label{lem:indice-2}
The sublattice $\Lb_X^G \oplus \Tb_X$ has index $2$ in $\Lb_{20}$.
\end{lemma}

\begin{proof}[Proof of Lemma~\ref{lem:indice-2}]
First, $\Lb_X^G \oplus \Tb_X$ is different from $\Lb_{20}$ 
since $\Lb_{20}$ is indecomposable 
(see Remark~\ref{rem:indec}). 
We have
$$\Lb_X^G=\{L \in \Lb_{20}~|~\iota(L)=L\},$$
$$\Tb_X=\{L \in \Lb_{20}~|~\iota(L)=-L\}.$$
By~\cite[Section 5]{nikulin}, the projection 
$\Lb_{20}/(\Lb_X^G  \oplus \Tb_X) \longto (\Lb_X^G)^\vee/\Lb_X^G$ 
is a $\iota$--invariant monomorphism. This shows in particular 
that $\Lb_{20}/(\Lb_X^G  \oplus \Tb_X)$ is cyclic. Also, if $L \in \Lb_{20}$, 
then
$$2L =\underbrace{L+\iota(L)}_{\in \Lb_X^G} + \underbrace{L-\iota(L)}_{\in \Tb_X} 
\in \Lb_X^G \oplus \Tb_X.$$
So the sublattice $\Lb_X^G \oplus \Tb_X$ has index $2$ in $\Lb_{20}$. 
This completes the proof of the Lemma.
\end{proof}
\end{quotation}

\medskip

We now come back to the proof of the theorem. We write $\Lb_{20}^G=\ZM L$. 
By the proof of Corollary~\ref{cor:parity}, we have $L^2=4n$ 
(so that $\Lb_{20}^G \simeq \langle 4n \rangle$) and 
the transcendental lattice of $X$ is of the form 
$$
\Tb_X=\left(\begin{array}{cc}
4a&2b\\
2b&4c\\
\end{array}
\right)
$$
with $a,b,c$ integers such that 
$d:=4ac-b^2>0$, $b^2\leq ac\leq \frac{d}{3}$, $-a\leq b\leq a\leq c$, 
see~e.g.~\cite[p. 128]{shiodainose}. We have shown in Lemma \ref{lem:indice-2} that $\Lb_{20}^G\oplus \Tb_X\simeq \langle 4n \rangle\oplus \Tb_X $ is a sublattice
of index $2$ in $\Lb_{20}$.  Hence we have by applying \cite[Section 2, Lemma 2.1]{BPV}
$$
4=[\Lb_{20}:\langle 4n \rangle\oplus \Tb_X]^2=\frac{\det (\langle 4n \rangle\oplus \Tb_X)}{\det \Lb_{20} }=\frac{16n(4ac-b^2)}{160}.
$$
In conclusion
$$
n(4ac-b^2)=2^3\cdot 5.
$$
We discuss two cases.

\medskip

\noindent{\bf Assume that $b$ is odd}. Then $4ac-b^2$ is also 
odd. This means that it is equal to $1$ or $5$, but then if $b=2k+1$ 
we get $4ac-4k^2-4k-1$ equal to $1$ or $5$ which is clearly impossible.

\medskip

\noindent{\bf Assume that $b$ is even}. Then with $b=2b'$ we get
$$
(ac-b'^2)n=2\cdot 5
$$
We distinguish four cases:
\begin{enumerate}
\item $n=1$, $ac-b'^2=10$,
\item $n=2$, $ac-b'^2=5$,
\item $n=5$, $ac-b'^2=2$,
\item $n=10$, $ac-b'^2=1$.
\end{enumerate} 
By Lemma~\ref{lem:plongements}, the lattices 
$\langle 4 \rangle$, $\langle 8 \rangle$ and $\langle 40 \rangle$ 
have a unique primitive embedding in the lattice $\Lb_{20}$: 
\begin{enumerate}
\item If $n=1$, we may assume that $L=e$. 
We now compute the orthogonal complement of $\IZ e$ in the lattice 
$\Lb_{20}$. This will give us the transcendental lattice. Let now 
$\lambda e+\mu f+\delta h$ with $\lambda,\mu,\delta\in\IZ$ be such that
$$
\langle \lambda e+\mu f+\delta h, e\rangle=0
$$
This gives $4\lambda- 2\delta=0$ so that the orthogonal complement is generated 
by the elements $e+2h$ and $f$ and considering instead the generators $e+f+2h$ and $f$ 
we get the lattice given in the theorem.

\item If $n=2$, we may assume that $L=e-f$. 
We compute the orthogonal complement of $e-f$ in $\Lb_{20}$ which is 
generated by $e+f$ and $-h$ which are the generators of the rank two lattice 
whose bilinear form is as given in the theorem. 

\itemth{4} If $n=10$, then the orthogonal complement of $L$ has been computed 
in~\cite{kondo} and one gets the rank two lattice 
whose bilinear form is given as in the theorem. 
\end{enumerate}
We have respectively $(a,b,c)=(1,0,1)$, $(a,b,c)=(1,0,10)$, 
$(a,b,c)=(2,2,3)$. 

\medskip

We consider now the third case with $ac-b'^2=2$ and we show 
that it is not possible. The integers $a, b, c$ satisfy $-a\leq b\leq a\leq c$, 
$ac\leq d/3$, $(b')^2\leq (ac)/4\leq d/3$.
By the previous computations, we have that 
$d=4(ac-b'^2)$ hence in this case $d=8$, we get that $b'^2\leq 2$. Hence $b'=0$
or $b'=1$. In the first case we get $a=1$, $c=2$ which gives the matrix 
$$
M:=\left(\begin{array}{cc}
4&0\\
0&8\\
\end{array}
\right).
$$
In the second case we get $a=1, c=3$ but then $ac=3>8/3$ so this is not possible. 
To make the case $\Tb_X=M$ possible, we should then find a primitive embedding in $\Lb_{20}$ with vectors $v_1$ and $v_2$ with $v_1^2=4$, $v_2^2=8$, $v_1\cdot v_2=0$ but by the computations in Lemma~\ref{lem:plongements} and with the same notations as there we see that we must send $v_1$ to $\pm e$ or $\pm f$ and $v_2$ to $\pm e\pm f$, 
so these
never satisfy the condition $v_1\cdot v_2=0$. 
% 
% The three K3 surfaces $X_{\kondo}$, $X_{\mukai}$ and $X_{\bra}$ 
% are described in the sections~\ref{sec:kondo}, \ref{sec:g29}, \ref{sec:q8}.
\end{proof}
%We recall here the classification theorem of S. Brandhort (see ??).
%\begin{theorem} \todo{I'm not sure what to do, maybe just talk about the three K3 without formulate a thm on they unicity ?
%We do not know now the statement in the theorem of Brandhorst....}
%Let $\G$ be a finite group acting on a K3 surface $X$, such that $\G$ contains
%the group $M_{20}$ which acts symplectically, then $[\G : M_{20}]\in \{2,4\}$. If $[\G : M_{20}]=2$ then 
%there are exactly two possibilities (up to isomorphism) of pairs $(X, \G)$. If $[\G : M_{20}]=4$ then
%there is only one pair $(X, \G)$. 

%The three K3 surfaces are isomorphic
%to the surfaces $X_\kondo$, $X_\mukai$, $X_\bra$. 
%\end{theorem}
%\medskip

\section{Kondo's example}\label{sec:kondo} 

\medskip

It was shown by Kondo in \cite[Theorem~1]{kondo} that the maximal  order of a finite group 
acting faithfully on a K3 surface is $3\,840$ and that this bound is reached for 
a unique K3 surface $X_{\kondo}$ and a unique faithful action of a unique finite group 
$G_\kondo$ of order $3\,840$. 
Kondo shows that $X_\kondo=\kummer(E_i\times E_i)$. 
Recall that we have an exact sequence
\equat\label{eq:kondo}
1\longto M_{20}\longto G_{\kondo}\longto \mub_4 \longto  1,
\endequat
where the last map is induced by the group homomorphism 
$$
\alpha: G_{\kondo} \longto  \IC^*,
$$
defined by $g(\omega_{X_\kondo})=\alpha(g)\omega_X$ and 
$\omega_{X_\kondo}$ is the holomorphic 
$2$-form that we have fixed on $X_\kondo$. Recall 
that $X_\kondo=\kummer(E_i\times E_i)$ (see e.g. \cite[Proof of Lemma 1.2]{kondo}) has transcendental lattice
$$
\Tb_{X_\kondo}=\begin{pmatrix}
4&0\\
0&4\\
\end{pmatrix}.
$$
With the previous notation we have:
\begin{pro}\label{quartic_kondo}
The invariant 
N\'eron--Severi group $NS(X_\kondo)^{M_{20}}=\mathbb{Z}\, L_{40}$ with $L_{40}^2=40$.
%Let $X$ be a K3 surface with a symplectic action by $M_{20}$. Then 
%the invariant lattice $\Lb_X^{M_{20}}$ is isometric to $\Lb_{20}$.
\end{pro}
\begin{proof}
See \cite[Lemma 3.1]{kondo}.
\end{proof}

\begin{remark}
In particular this means that we cannot represent $X_\kondo$ as a quartic 
surface in $\mathbb{P}^3(\mathbb{C})$ with a faithful action
of $M_{20}$ by linear transformations of $\mathbb{P}^3(\mathbb{C})$.\finl
\end{remark}

% Assume the contrary  that $Km(E_i\times E_i)\cong X_\mukai$. Then one can choose a polarization of degree $4$, i.e. an ample divisor $L_4$ on 
% $Km(E_i\times E_i)$ such that $L_4^2=4$ which gives an embedding of $Km(E_i\times E_i)$ in 
% $\IP^3$ as the smooth quartic $X_\mukai$ and $L_4$ is invariant by 
% $M_{20}$. Recall that $Km(E_i\times E_i)$ has transcendental lattice
% $$
% T=\left(\begin{array}{cc}
% 4&0\\
% 0&4\\
% \end{array}
% \right).
% $$
% Kondo in \cite{kondo} showed that the invariant ample cone for the action of $M_{20}$ is one dimensional generated by an ample class of square $40$. In fact the N\'eron-Severi group contains the orthogonal complement to the invariant lattice in the K3 lattice for the action of $M_{20}$ which has rank 19, this by a result of Nikulin is negative definite (and it contains no element of square $-2$) so that since our K3 surface is projective, the N\'eron-Severi group contains an ample invariant polarization, which spans then the invariant ample cone and this is 1--dimensional. In \cite{kondo}, Kondo computed that the invariant  lattice for the action of $M_{20}$ is isometric to
% $$
% L^{M_{20}}:=\left(\begin{array}{ccc}
% 4&0&-2\\
% 0&4&-2\\
% -2&-2&12
% \end{array}
% \right)
% $$
% and that the orthogonal complement of $T$ is then the lattice $\langle 40 \rangle$, so there is no invariant polarization of degree $4$.

\subsection{A geometric model}
By using a result of Inose \cite[Theorem 2]{inose} one can view 
$X_\kondo=\kummer(E_i\times E_i)$ as the minimal resolution of a 
singular surface in $\IP(1,1,2,2,2)$. We give here the equation.   
Inose shows that $X_\kondo$ is the minimal resolution of the quotient of the Fermat 
quartic surface
$$
F:\,\,x^4+y^4+z^4+t^4=0
$$
by the symplectic involution $\iota: (x:y:z:t)\mapsto (x:y:-z:-t)$, which 
has 8 isolated fixed points \cite[Section 5]{nikulin1}. Since the automorphism 
is symplectic, the minimal resolution of the quotient 
$X_\kondo \to F/\langle \iota\rangle$ is again a K3 surface
and the Picard number remains unchanged. Moreover, for the 
transcendental lattices $\Tb_{X_\kondo}(2)=\Tb_F$ holds.  
The ring
of invariant polynomials for the action of $\iota$ is 
generated by $x,y,z^2,t^2,zt$. We put $z_0=x, z_1=y,
 z_2=z^2, z_3=t^2, z_4=zt$ and we have then the 
 equations for $F/\langle \iota \rangle$ in $\IP(1,1,2,2,2)$:
$$
z_0^4+z_1^4+z_2^2+z_3^2=0,\,\, z_4^2=z_2z_3.
$$
The eight $A_1$ singularities are determined as follows. First 
we have singularities coming from the ambient space, these are the intersection with
the plane $z_0=z_1=0$. This gives $z_2^2+z_3^2=0$ which together with $z_4^2=z_2z_3$ gives four $A_1$ singularities. The others come from the singularities of the cone  $z_4^2=z_2z_3$, i.e. with $z_4=z_2=z_3=0$ we get the four singularities $A_1$ with equation $z_0^4+z_1^4=0$.

See also~\cite{BH} for an embedding of $X_\kondo$ in $\PM^{21}(\CM)$.

\bigskip

\section{Mukai's example}\label{sec:g29}

\medskip

Let $G_\mukai=\langle s_1,s_2,s_3,s_4\rangle$, where 
$$
\begin{array}{cc}
s_1=\begin{pmatrix}
1  &0 & 0 & 0\\
0  &1 & 0 & 0\\
0 & 0 & 1 & 0\\
0 & 0 & 0 &-1\\
\end{pmatrix}, &
s_2=\DS{\frac{1}{2}}\begin{pmatrix}
 1&  1&  i&  i\\
 1 & 1& -i& -i\\
-i&  i&  1& -1\\
-i  &i &-1&  1\\
\end{pmatrix},\\
&\\
s_3=\begin{pmatrix}
 0 &1& 0& 0\\
 1 &0 &0 &0\\
 0 &0 &1& 0\\
 0& 0 &0& 1\\
\end{pmatrix},&
s_4=\begin{pmatrix}
 1 &0& 0& 0\\
 0 &0 &1 &0\\
 0& 1 &0 &0\\
 0 &0 &0 &1\\
\end{pmatrix}. \\
\end{array}
$$
Then $G_\mukai$ is the primitive 
complex reflection group denoted by $G_{29}$ in Shephard-Todd 
classification~\cite{shephardtodd}. Recall that $|G_\mukai|=7~\!680$. 
We denote by $V$ the vector space $\CM^4$, and by $\CM[V]$ 
the algebra of polynomial functions on $V$, 
identified naturally with $\CM[x,y,z,t]$. 
If $m$ is a monomial in $x$, $y$, $z$ and $t$, 
we denote by $\Sigma(m)$ the sum of all monomials 
obtained by permutation of the variables. For instance, 
$$\Sigma(x)=x+y+z+t,\qquad \Sigma(xyzt)=xyzt,$$
$$\Sigma(x^4y)=x^4(y+z+t)+y^4(x+z+t)+z^4(x+y+t)+t^4(x+y+z)=
\Sigma(xy^4).$$

\medskip

Note that the derived subgroup $G_\mukai'$ of $G_\mukai$ 
has index $2$, that 
$G_\mukai'=G_\mukai \cap \Sb\Lb_4(\CM)$, 
so that $G_\mukai= G_\mukai'\langle s_1 \rangle$. 
Note also that $\Zrm(G_\mukai) \simeq \mub_4 \subset G_\mukai'$. 
Moreover, $PG_\mukai' \simeq M_{20}$ 
so that we have a split exact sequence 
\equat\label{eq:pg29}
1 \longto PG_\mukai' \simeq M_{20} \longto PG_\mukai \longto \mub_2 \longto 1,
\endequat
where the last map is the determinant. 

Now, there exists a unique (up to scalar) 
homogeneous invariant $f$ of $G_\mukai$ of degree $4$: it is given by  
$$f=\Sigma(x^4)-6\Sigma(x^2y^2).$$
We set $X_\mukai=\ZC(f)$. 
It can easily be checked that $X_\mukai$ is a 
smooth and irreducible quartic in $\PM^3(\CM)$, so that it is a K3 surface, 
endowed with a faithful symplectic action of $M_{20}$ and an extra non-symplectic 
automorphism of order $2$, i.e. one can fix it as 
$[x:y:z:t]\mapsto [x:y:z:-t]$, the one induced by $s_1$. 

In \cite[nr. 4 on p. 190]{mukai} Mukai gives the following equation for some $M_{20}$-invariant quartic polynomial
$$
\Sigma(x^4)+12xyzt,
$$
and we denote by $X_\mukai'$ the zero set of this polynomial which defines a smooth
quartic K3 surface. We have

\begin{pro}
There exists $g \in \Gb\Lb_4(\CM)$ such that $g(X_\mukai) = X_\mukai'$.
\end{pro}

\begin{proof}
If one applies to the Mukai's polynomial the change of coordinates:
$$
x\mapsto x-y,\,y\mapsto x+y,\, z\mapsto z-t,\, t\mapsto z+t 
$$
one gets 
$$
2\Sigma(x^4)+12x^2y^2+12z^2t^2+12x^2z^2-12x^2t^2-12y^2z^2+12y^2t^2
$$
and by replacing by
$$
x\mapsto ix,\, t\mapsto it,\, y\mapsto y,\,  
$$
and dividing by $2$ one finds the polynomial $f$. 
\end{proof}

%As explained in~\cite{BS2}, $X_\mukai$ is the quartic defined by Mukai 
%in~\cite{mukai}: Mukai gives the following equation for some 
%$M_{20}$-invariant quartic polynomial
%byt an easy linear change of coordinates show that it defines 
%the same K3 surface as $X_\mukai$.

\medskip

Note the 
following fact:
\equat\label{eq:aut q4}
\text{\it If $g\in \Pb\Gb\Lb_4(\IC)$ leaves invariant $X_\mukai$ then  
$g\in PG_\mukai$.}
\endequat
\begin{proof}
If $g \in \Pb\Gb\Lb_4(\CM)$ leaves $X_\mukai$ invariant, we 
may find a representative $\gti$ of $g$ in $\Gb\Lb_4(\CM)$ which leaves $f$ invariant. 
Let $\G=\{\g \in \Gb\Lb_4(\CM)~|~\lexp{\g}{f}=f\}$. 
We only need to prove that $\G=G_\mukai$. 
By~\cite{mm} or~\cite[Theorem~2.1]{orlik-solomon}, 
$\G$ is finite (because $X_\mukai$ is smooth), 
and contains $G_\mukai$. Let $R$ denote the set of reflections in 
$G_\mukai$ (and recall that $G_\mukai=\langle R \rangle$) 
and let 
$$\RC=\{\g s \g^{-1}~|~\g \in \G~\text{and}~s \in R\},$$
so that $\RC$ is a set of reflections contained in $\G$. 
We set $\G_\RC = \langle \RC \rangle$. Then $\G_\RC$ is a complex 
reflection group containing $G_\mukai$, but it follows from the 
classification of primitive complex reflection groups that 
$\G_\RC=G_\mukai$ or (up to conjugacy) the group denoted by $G_{31}$ 
in Shephard-Todd classification~\cite{shephardtodd}. Since $G_{31}$ 
has no non-zero invariant of degree $4$, this forces 
$\G_\RC=G_\mukai$. In particular, $G_\mukai$ is normal in $\G$, 
and so the result follows from~\cite[Proposition~3.13]{BMM} 
(which says that $\Nrm_{\Gb\Lb_4(\CM)}(G_\mukai)=G_\mukai\cdot \CM^\times$). 
\end{proof}

\medskip

The embedding $X_\mukai \injto \PM^3(\CM)$ defines the class of a hyperplane section 
on $X_\mukai$ that we denote by $L_4$: then $L_4^2=4$ and $L_4$ is $PG_\mukai$-invariant.

\medskip

\begin{pro}\label{degree4}
With the above notation, we have:
\begin{enumerate}
\itemth{1} The transcendental lattice of $X_\mukai$ is a rank 
two lattice given by 
$$
\Tb_{X_\mukai}:=\left(\begin{array}{cc}
4&0\\
0&40\\
\end{array}
\right)
$$ 
and $NS(X_\mukai)^{M_{20}}=\mathbb{Z}\, L_4$ with $L_4^2=4$.

\itemth{2} The quartic  $X_\mukai$ is the unique invariant quartic 
for a faithful action of $M_{20}$ on $\IP^3$. 
\end{enumerate}
\end{pro}
\begin{proof}
(1) has been proved in Theorem~\ref{involutions}, see also \cite[Section 3]{dolga_quartic}. 

\medskip

(2) Let $Q \in \PM^3(\CM)$ be a quartic leaved invariant by a faithful  
action of $M_{20}$. This means that there exists a representation 
of $M_{20}$ as a subgroup of 
$\Pb\Gb\Lb_4(\CM)$ which stabilizes $Q$. Then $Q$ is polarized 
by the lattice $\langle 4 \rangle $, so that we have an embedding of  
$\langle 4 \rangle $ in the lattice $\Lb_{Q}^{M_{20}}$. 
Since this embedding is unique by~(1), its orthogonal 
complement $\Tb_Q$ in $\Lb_{Q}^{M_{20}}$ is isometric to 
$\Tb_{X_\mukai}$. So $Q$ is projectively equivalent to $X_\mukai$.
\end{proof}

\medskip

\begin{pro}
The quartic $X_\mukai$ is the Kummer surface 
$\kummer(E_{i\sqrt{10}}\times E_{i\sqrt{10}})$.
\end{pro}

\medskip

\begin{proof}
This follows from Corollary~\ref{cor:parity} and its proof.
\end{proof}

\medskip
% 
% \begin{remark}
% Moreover by \cite[Theorem 5.1]{shiodamitani} if $A'$ is another abelian surface 
% such that $\kummer(A')=\kummer(E_{i\sqrt{10}}\times E_{i\sqrt{10}})$ then 
% $A'$ is isomorphic to 
% $E_{i\sqrt{10}}\times E_{i\sqrt{10}}$.\finl \todo{Est-ce qu'on ne pourrait pas mettre l'unicite une fois pour toutes dans le corollaire~\ref{cor:parity} ?}
% \end{remark}
% 
% \medskip

\begin{remark}\label{rem:coniques dans q4}
As $X_\mukai$ is a Kummer surface, it admits $16$ two by two 
disjoint smooth rational curves (a {\it Nikulin configuration}). 
We were not able to find such a set of smooth rational curves, 
but, using {\sc Magma}, we have at least found $320$ conics in $X_\mukai$ 
(from which it is impossible 
to extract a Nikulin configuration: we can only extract 
$12$ two by two disjoint conics). Let
$$C_+=\{[x:y:z:t] \in \PM^3(\CM)~|~x+y+z=y^2 + yz + z^2 + 
\frac{3+\sqrt{10}}{2} t^2=0\}$$
$$C_-=\{[x:y:z:t] \in \PM^3(\CM)~|~x+y+z=y^2 + yz + z^2 + 
\frac{3-\sqrt{10}}{2} t^2=0\}.
\leqno{\text{and}}$$
Then $C_+$ and $C_-$ are two smooth conics contained in 
$X_\mukai$ and, if we denote 
by $\O_\pm$ the $G_\mukai$-orbit of $C_\pm$, then $\O_+ \neq \O_-$, 
$|\O_\pm|=160$, and all elements of $\O_\pm$ are 
contained in $X_\mukai$.\finl
\end{remark}

\medskip

\begin{remark}
Observe that $PG_\mukai$ is a maximal finite subgroup of $\Aut(X_\mukai)$. 
Indeed, if $PG_\mukai \varsubsetneq \G \subset \Aut(X_\mukai)$ with $\G$ finite, 
then $|\G| \ge 2\cdot |PG_\mukai| = 3~\!840$ and so by the result of Kondo in~\cite{kondo} the group $\G$ would be the group 
$G_\kondo$ defined in section~\ref{sec:kondo} 
and $X_\mukai$ would be isomorphic to $X_\kondo$: this is not the 
case by Proposition~\ref{quartic_kondo} and Proposition~\ref{degree4}.\finl 
\end{remark}

\medskip

\section{Brandhorst-Hashimoto's example}\label{sec:q8}

\medskip

Let $G_\bra$ be the subgroup of $\Gb\Lb_6(\CM)$ generated by 
$$t=\diag(-1,1,1,1,1,1),$$
$$u=
\begin{pmatrix}
i  &  0  &  0  &  0  &  0  &  0 \\
0  &  0  &  1  &  0  &  0  &  0 \\
0  &  1  &  0  &  0  &  0  &  0 \\
0  &  0  &  0  & -i  &  0  &  0 \\
0  &  0  &  0  &  0  &  0  &  1 \\
0  &  0  &  0  &  0  &  1  &  0 \\
\end{pmatrix}\quad\text{and}\quad v=
\begin{pmatrix}
0 & 0 & 0 & 0 & 0 & 1 \\
1 & 0 & 0 & 0 & 0 & 0 \\
0 & 1 & 0 & 0 & 0 & 0 \\
0 & 0 & 1 & 0 & 0 & 0 \\
0 & 0 & 0 & 0 & 1 & 0 \\
0 & 0 & 0 & 1 & 0 & 0 \\
\end{pmatrix}.
$$
All the numerical facts about $G_\bra$ stated below can be checked with 
{\sc Magma}. 
Then $|G_\bra|=3~\!840$, $\Zrm(G)=\mub_2$, $|G_\bra/G_\bra'|=2$ 
and there are two exact sequences 
$$1 \longto \mub_2 \longto G_\bra' \longto M_{20} \longto 1$$
and 
\equat\label{eq:brandhorst}
1 \longto M_{20}=PG_\bra' \longto PG_\bra \longto \mub_2 \longto 1.
\endequat
The second exact sequence splits (for instance by sending the non-trivial 
element of $\mub_2$ to $t$) and $G_\bra'=G_\bra \cap \Sb\Lb_6(\CM)$. 
Even though the last exact sequence looks like~(\ref{eq:pg29}), 
\equat\label{eq:non-iso}
\text{\it The groups $PG_{\mukai}$ and $PG_\bra$ are not isomorphic.}
\endequat
Note that 
the group $G_\bra'$ is isomorphic to the group denoted by $2_3.M_{20}$ 
in the {\sc Atlas} of finite 
groups\footnote{http://brauer.maths.qmul.ac.uk/Atlas/v3/group/M20/}. 
We denote by $V = \CM^6$ the natural representation of $G_\bra$ 
and we identify $\CM[V]$ with $\CM[x_1,x_2,x_3,x_4,x_5,x_6]$. 
Note that 
\equat\label{eq:double-transitif}
\text{\it $G_\bra$ acts doubly transitively on the set of hyperplanes  
$\{H_1,\dots,H_6\}$,}
\endequat
where $H_i$ is defined by $x_i = 0$. 

S. Brandhorst and K. Hashimoto~\cite{BH} proved that there is a unique K3 surface 
admitting a faithful action of $PG_\bra$ and, in a private 
communication, they asked the question about the equations of this 
K3 surface: 
the aim of 
this section is to answer the question by exhibiting 
explicit equations of such a K3 surface. 

The group $G_\bra$ contains the group $N$ of diagonal matrices with coefficients 
in $\mub_2$ as a normal subgroup (so $N \simeq (\mub_2)^6$) 
and we have $G_\bra/N \simeq \AG_5$. It is easy to see that
\equat\label{eq:cvn}
\CM[V]^N=\CM[x_1^2,x_2^2,x_3^2,x_4^2,x_5^2,x_6^2].
\endequat
The following facts are checked with {\sc Magma}:
\begin{itemize}
\itemth{a} As a $G_\bra/N$-module, $\CM[V]_2^N = S_1 \oplus S_2$, where 
$S_1$ and $S_2$ are the two non-isomorphic irreducible representations 
of $G_\bra/N \simeq \AG_5$ of dimension $3$. 

\itemth{b} Let $\phi=(1+\sqrt{5})/2$ be the golden ratio. If we set 
$$
\begin{cases}
q_1=x_1^2 + x_4^2 -\phi x_5^2 +  \phi x_6^2,\\
q_2=x_2^2 - \phi x_4^2 + x_5^2 -\phi x_6^2,\\
q_3=x_3^2 + \phi x_4^2 - \phi x_5^2  + x_6^2,
\end{cases}$$
then $(q_1,q_2,q_3)$ is a basis of $S_1$.
\end{itemize}
We then define
$$X_\bra=\ZC(q_1,q_2,q_3).$$
The next proposition can be proved using {\sc Magma}, but we will 
provide a proof independent of {\sc Magma} computations.

\smallskip

\begin{pro}\label{pro:x-bra}
The scheme $X_\bra$ is smooth, irreducible, of dimension $2$.
\end{pro}

\smallskip

The variety $X_\bra$ is then an irreducible 
smooth complete intersection 
of three quadrics in $\PM^5(\CM)$, so it is a K3 surface. Since the vector space 
$S_k$ is stable under the action 
of $G_\bra$, the K3 surface $X_\bra$ is endowed with a faithful action of 
$PG_\bra \simeq \langle t \rangle \ltimes  M_{20}$.

\smallskip

\begin{cor}\label{cor:x-bra}
$X_\bra$ is a K3 surface endowed with a faithful action of $PG_\bra$. 
\end{cor}

We show first the following:

\begin{pro}\label{quot_bra}
Let $H=N\cap G'_{BH}$, then the scheme $X_\bra/H$ is a K3 surface (with $A_1$ singularities) which is a double cover of $\mathbb{P}^2(\mathbb{C})$
 ramified on the union of $6$ lines in general position.
\end{pro}
\begin{proof}
Note that 
$$\CM[x_1,x_2,x_3,x_4,x_5,x_6]^H = 
\CM[x_1^2,x_2^2,x_3^2,x_4^2,x_5^2,x_6^2,x_1x_2\cdots x_6],$$
so that 
$\PM^5(\CM)/H=\{[y_1:\cdots:y_6:z] \in \PM(1,\dots,1,3)~|~z^2 =\prod_{k=1}^6 y_k\}$. 
Therefore,
$$X_\bra/H = \{[y_1:\cdots:y_6:z] \in \PM(1,\dots,1,3)~|~z^2 =\prod_{k=1}^6 y_k$$
$$~\text{and}~
\begin{cases}
y_1+y_4-\phi y_5 + \phi y_6 = 0 \\
y_2-\phi y_4 + y_5 -\phi y_6 = 0 \\
y_3 + \phi y_4 - \phi y_5 + y_6=0\\ 
\end{cases} \quad\}$$
Simplifying the equations, one gets
$$
X_\bra/H =\{[y_4:y_5:y_6:z] \in \PM(1,1,1,3)~|~\hphantom{AAAAAAAAAAAAAAAAA}\\$$
$$\hphantom{AAAAAA}
z^2 =y_4y_5y_6(-y_4+\phi y_5 - \phi y_6)
(\phi y_4 - y_5 +\phi y_6)(-\phi y_4 + \phi y_5 - y_6)\}. 
$$
So $X_\bra/H$ is a K3 surface (with $A_1$ singularities) 
which is a double cover of $\PM^2(\CM)$ 
ramified on the union of $6$ lines in general position as claimed.  
\end{proof}

\smallskip

\begin{proof}[Another proof of Proposition~\ref{pro:x-bra}]
 First, it follows from~(\ref{eq:cvn}) that 
\equat\label{eq:xbran}
X_\bra/N = \{[y_1:\cdots:y_6] \in \PM^5(\CM)~|~
\begin{cases}
y_1+y_4-\phi y_5 + \phi y_6 = 0 \\
y_2-\phi y_4 + y_5 -\phi y_6 = 0 \\
y_3 + \phi y_4 - \phi y_5 + y_6=0\\ 
\end{cases} \} \simeq \PM^2(\CM).
\endequat
Hence $X_\bra/N$ has dimension $2$, so $X_\bra$ has dimension $2$. 
Then one can use \cite[Exercice III, 5.5]{hartshorne} to see that $X_\bra$ is connected, so that if it is smooth then it is irreducible. 
We prove smoothness below, but we can also argue in the way as follows.

\medskip

By Proposition \ref{quot_bra} the quotient $X_\bra/H$ is irreducible. This shows that $H$ acts transitively 
on the irreducible components of $X_\bra$. So $G_\bra'$ also 
acts transitively on the irreducible components. Now, let $X$ be an 
irreducible component 
of $X_\bra$ and let $K$ denote its stabilizer in $G_\bra'$. 
Then $8=\deg(X_\bra)=\deg(X) \cdot |G_\bra'/K|$. 
Since $G_\bra'$ has no subgroup of index $2$, $4$ or $8$, 
we conclude that $K=G_\bra'$, so that $X=X_\bra$, as desired.

\medskip

We now show that $X_\bra$ is smooth. Let $p=[x_1:x_2:x_3:x_4:x_5:x_6] \in X_\bra$ 
and assume that $p$ is a singular point of $X_\bra$. Since $p$ belongs to $X_\bra$, 
the equations show that at least two of the $x_k$'s are non-zero. By replacing 
if necessary $p$ by another point in its $G_\bra$-orbit, we may assume that 
$x_1x_2 \neq 0$ (thanks to~(\ref{eq:double-transitif})). 
The Jacobian matrix of $(q_1,q_2,q_3)$ at $p$ is given by
$${\mathrm{Jac}}_p(q_1,q_2,q_3)=
\begin{pmatrix}
2x_1 & 0 & 0& 2x_4 & -2\phi x_5 & 2\phi x_6 \\
0& 2x_2 &0 & -2\phi x_4 & 2x_5 & -2\phi x_6 \\
0&0& 2x_3 & 2\phi x_4 & -2\phi x_5 & 2 x_6 \\
\end{pmatrix}.$$
Then the rank 
of ${\mathrm{Jac}}_p(q_1,q_2,q_3)$ is less than $3$, which means 
that all its minors of size $3$ vanish. Therefore,
$$x_{i_1}x_{i_2}x_{i_3}=0$$
for all $1 \le i_1 < i_2 < i_3 \le 6$. Since $x_1x_2 \neq 0$, 
we get $x_3=x_4=x_5=x_6=0$. But then $q_1(p) \neq 0$, 
which is impossible.
\end{proof}

\bigskip

\begin{remark}\label{rem:x-x'}
Exchanging $S_1$ and $S_2$ (whose characters are Galois conjugate 
under $\sqrt{5} \mapsto -\sqrt{5}$), one gets another K3 surface $X_\bra'$, 
where $\phi$ is replaced by its Galois conjugate $\phi'=(1-\sqrt{5})/2=1-\phi$ 
in the equations. 
Let $\s \in \Gb\Lb_6(\CM)$ be the matrix 
$$\s=
\begin{pmatrix}
1 & 0 & 0 & 0 & 0 & 0 \\
0 & 0 & 0 & 0 & 0 & -i \\
0 & 0 & i & 0 & 0 & 0 \\
0 & 1 & 0 & 0 & 0 & 0 \\
0 & 0 & 0 & i & 0 & 0 \\
0 & 0 & 0 & 0 & 1 & 0 \\
\end{pmatrix}.$$
Then $\s$ normalizes $G_\bra$ and $\s(X_\bra)=X_\bra'$, 
so that $X_\bra$ and $X_\bra'$ are isomorphic.\finl
\end{remark}

\bigskip

The surface $X_\bra$ is a K3 surface with polarization $L_8$ 
satisfying $L_8^2=8$, 
and as in section \ref{sec:g29} this is invariant by the action of $M_{20}$. 
We have hence an embedding of $\langle 8 \rangle$ in $\Lb_{X_\bra}^{M_{20}}$. 

\bigskip

\begin{pro}\label{prop:q8-reseau}
With the above notation, we have:
\begin{enumerate}
\itemth{1} The transcendental lattice of $X_\bra$ is a 
rank two lattice 
given by
$$
\Tb_{X_\bra}=\left(\begin{array}{cc}
8&4\\
4&12\\
\end{array}
\right)
$$
and $NS(X_\bra)^{M_{20}}=\mathbb{Z}\, L_8$ with $L_8^2=8$.

\itemth{2} The complete intersection $X_\bra$ is the unique K3 surface 
invariant for a faithful action of $M_{20}$ in $\IP^5(\CM)$. 
\end{enumerate}
\end{pro}

\begin{proof}
(1) has been proved in~Theorem~\ref{involutions}.

\medskip

(2) follows from the same argument as in Proposition \ref{degree4}. 
\end{proof}

\medskip

\begin{remark}
Proposition~\ref{prop:q8-reseau} gives another proof that 
$X_\bra\cong X_\bra'$.\finl
\end{remark}

\medskip

\begin{pro}
The K3 surface $X_\bra$ is the Kummer surface $\kummer(E_{\tau}\times E_{2\tau})$, 
with $\tau_1=\frac{-1+i\sqrt{5}}{2}$.
\end{pro}

\medskip

\begin{proof}
This follows from Corollary~\ref{cor:parity} and its proof.
\end{proof}

\medskip
% 
% \begin{remark}\label{rem:unicite}
% By \cite[Theorem 5.1]{shiodamitani}, 
% if $A'$ is another abelian surface such that 
% $\kummer(A')=\kummer(E_{\tau_1}\times E_{\tau_2})$ then $A'$ is isomorphic to 
% $E_{\tau_1}\times E_{\tau_2}$.\finl
% \end{remark}
% 
% \medskip

\begin{remark}[Smooth rational curves]\label{curves bra}
Using {\sc Magma}, one can find an explicit 
Nikulin configuration in $X_\bra$ as follows. 
Let $C$ denote the conic defined by the equations
$$\begin{cases}
x_5=\sqrt{\phi} x_1,\\
x_4=\sqrt{\phi} x_2,\\
x_3=\sqrt{\phi} x_6,\\
x_1^2-x_2^2-x_6^2=0
\end{cases}$$
and let $\AC$ denote the subgroup of $G_\bra$ generated by 
$${\tiny\begin{pmatrix}
0 & 0 & 1 & 0 & 0 & 0 \\
0 & 0 & 0 & 0 & i & 0 \\
-1& 0 & 0 & 0 & 0 & 0 \\
0 & 0 & 0 & i & 0 & 0 \\
0 & i & 0 & 0 & 0 & 0 \\
0 & 0 & 0 & 0 & 0 & -i \\
\end{pmatrix},\quad
\begin{pmatrix}
0 & 0 &-i & 0 & 0 & 0 \\
0 &-i & 0 & 0 & 0 & 0 \\
-i& 0 & 0 & 0 & 0 & 0 \\
0 & 0 & 0 & 0 & 0 &-i \\
0 & 0 & 0 & 0 & i & 0 \\
0 & 0 & 0 &-i & 0 & 0 \\
\end{pmatrix},}$$
$${\tiny\begin{pmatrix}
1 & 0 & 0 & 0 & 0 & 0 \\
0 & 1 & 0 & 0 & 0 & 0 \\
0 & 0 & 1 & 0 & 0 & 0 \\
0 & 0 & 0 &-1 & 0 & 0 \\
0 & 0 & 0 & 0 & 1 & 0 \\
0 & 0 & 0 & 0 & 0 &-1 \\
\end{pmatrix}\quad\text{\normalsize and}\quad
\begin{pmatrix}
1 & 0 & 0 & 0 & 0 & 0 \\
0 &-1 & 0 & 0 & 0 & 0 \\
0 & 0 & 1 & 0 & 0 & 0 \\
0 & 0 & 0 &-1 & 0 & 0 \\
0 & 0 & 0 & 0 &-1 & 0 \\
0 & 0 & 0 & 0 & 0 &-1 \\
\end{pmatrix}}
$$
Then $C$ is contained in $X_\bra$. It can be checked with {\sc Magma} 
that its $G_\bra$-orbit contains $80$ elements, and that 
its $\AC$-orbit contains $16$ elements which are two by two
disjoint (note that $|\AC|=32$, that $\mub_2 \subset \AC$ 
and that $\AC/\mub_2$ is elementary abelian).

Note also that the conic defined by the equations 
$$\begin{cases}
x_1 + ix_5 - i\phi x_6=0,\\
x_3 - i\phi x_5 + i\phi x_6=0,\\
x_4 - \phi x_5 + x_6=0,\\
x_2^2 - 2\phi x_5^2 + 2(1+\phi)x_5 x_6 - 2\phi x_6^2=0,
\end{cases}$$
is contained in $X_\bra$, and that its $G_\bra$-orbit contains 
$96$ elements. However, we can only extract subsets of $12$ 
two by two disjoint conics from this orbit.\finl
\end{remark}

\medskip

\section{Final Remarks}

\medskip

\begin{pro}
The K3 surfaces $X_\mukai$, $X_\bra$ and $X_\kondo$ 
are two by two non-isomorphic.
\end{pro}

\medskip

\begin{proof}
Indeed, they do not have the same transcendental lattice 
(or equivalently they do not admit
polarizations of the same degree).
\end{proof}

\medskip

\begin{pro}\label{unicity}
If a K3 surface $X$ admits a faithful action of $G_{\kondo}$, $PG_{\mukai}$, 
respectively $PG_{\bra}$ then $X$ is isomorphic to $X_{\kondo}$, $X_{\mukai}$, 
respectively $X_{\bra}$.
\end{pro}

\medskip

\begin{proof}
For $G_{\kondo}$ this is shown in \cite[Lemma 3.1]{kondo}. Before going on, 
note the following fact, which can easily be checked with {\sc Magma}:
\equat\label{eq:inclusion}
\text{\it The groups $PG_\mukai$ and $PG_\bra$ are not isomorphic to subgroups of 
$G_\kondo$.}
\endequat
Consider now the group 
$G_{\mukai}$, then $PG_{\mukai}/M_{20}=\langle \iota \rangle$ and 
$\iota$ acts non-symplectically, hence $X$ is one of the three surfaces of 
Theorem \ref{involutions} and $PG_{\mukai}$ leaves invariant
the polarization, hence it is realized by linear transformations. 
We only need to show that $X_{\kondo}$ and $X_{\bra}$ do not admit
an automorphism group isomorphic to $PG_{\mukai}$. Assume it is the case, 
then $PG_{\mukai}$ and $G_{\kondo}$ leaves invariant the polarization of 
degree $\langle 40 \rangle$ on $X_{\kondo}$, hence by~\cite[Proposition 5.3.3]{Huybrechts} the group that they generate together is finite. By the maximality of $G_{\kondo}$ this means that 
$PG_{\mukai}$ is contained in $G_{\kondo}$ but by~(\ref{eq:inclusion}) 
the group $G_{\kondo}$ does not contain such a subgroup. With a similar argument 
if $PG_{\mukai}$  acts on $X_{\bra}$ then we conclude that $PG_{\mukai}\cong PG_{\bra}$ and this is not the case by~(\ref{eq:non-iso}). 
The same argument holds 
for $PG_{\bra}$. 
\end{proof}

\medskip

\begin{theorem}\label{max}
Let $G$ be a maximal finite group with a faithful and non--symplectic action on
a K3 surface 
$X$ and assume that $M_{20}\subset G$. Then $G$ is isomorphic
to $G_{\kondo}$, $PG_{\mukai}$ or $PG_{\bra}$.
\end{theorem}

\medskip

\begin{proof}
Since $G$ acts non-symplectically then $G/M_{20}$ is non-trivial 
and by \cite{kondo} it has order at most four. 
If $|G/M_{20}|=4$ then
$G\cong G_{\kondo}$ by \cite{kondo}. Observe that the group $G/M_{20}$ 
acts faithfully on $\Lb_{20}$ since it contains $\Tb_X$. 
By Remark~\ref{rem:l20}, 
the group of isometries of $\Lb_{20}$  has order $2^4$ so it is not 
possible to have $|G/M_{20}|=3$. We are left with the case 
$|G/M_{20}|=2$. By Theorem \ref{involutions} the K3 surface $X$ is isomorphic to $X_{\kondo}$, $X_{\mukai}$ or $X_{\bra}$. By the same argument as 
in Proposition \ref{unicity} and the maximality of $G$, then $G$ is isomorphic to $PG_{\mukai}$ or $PG_{\bra}$.  
\end{proof}

\bibliographystyle{amsplain}
\bibliography{Biblio}

\end{document}